\documentclass[final]{amsart}

\usepackage{amsthm,amsfonts,amssymb,amsmath,latexsym, amscd, mathabx,mathrsfs, oldgerm,float}
\usepackage{exscale, cite, epsfig, graphics}
\usepackage{color}
\usepackage[dvipsnames]{xcolor}
\usepackage{tikz}

\usepackage{accents}
\renewcommand*{\dot}[1]{%
  \accentset{\mbox{\bfseries .}}{#1}}

\usepackage[notref, notcite]{showkeys}
\usepackage[colorlinks]{hyperref}

\renewcommand{\geq}{\geqslant}
\renewcommand{\leq}{\leqslant}

\newtheorem{theorem}{Theorem}[section]
\newtheorem{lemma}[theorem]{Lemma}

\newtheorem{corollary}[theorem]{Corollary}

\newtheorem*{main-theorem}{Main Theorem}
\newtheorem*{remark*}{Remark}
\numberwithin{equation}{section}

\title[Norm inflation for fractional KdV equations]
{Norm inflation for equations of KdV type \\ with fractional dispersion}

\author[Hur]{Vera~Mikyoung~Hur}
\address{Department of Mathematics, University of Illinois at Urbana-Champaign, Urbana, IL 61801 USA}
\email{verahur@math.uiuc.edu}

\date{\today}


\begin{document}

\maketitle

\begin{abstract}
We demonstrate norm inflation for nonlinear nonlocal equations, which extend the Korteweg-de Vries equation to permit fractional dispersion, in the periodic and non-periodic settings. That is, an initial datum is smooth and arbitrarily small in a Sobolev space, but the solution becomes arbitrarily large in the Sobolev space after an arbitrarily short time. 
\end{abstract}

%%%%%%%%%%%%%%%%%%%%%%%%%%%%%%%%%%%%%%%%%%%%%%%%%%
%%%%%%%%%%%%%%%%%%%%%%%%%%%%%%%%%%%%%%%%%%%%%%%%%%
%%%%%%%%%%%%%%%%%%%%%%%%%%%%%%%%%%%%%%%%%%%%%%%%%%
\section{Introduction}\label{sec:intro}
%%%%%%%%%%%%%%%%%%%%%%%%%%%%%%%%%%%%%%%%%%%%%%%%%%
%%%%%%%%%%%%%%%%%%%%%%%%%%%%%%%%%%%%%%%%%%%%%%%%%%
%%%%%%%%%%%%%%%%%%%%%%%%%%%%%%%%%%%%%%%%%%%%%%%%%%

We address the ill-posedness of the Cauchy problem associated with equations of Korteweg-de Vries type:
\begin{equation}\label{E:main}
\partial_tu+|\partial_x|^\alpha \partial_xu+u\partial_xu=0
\end{equation}
and 
\begin{equation}\label{def:u0}
u(x,0)=u_0(x).
\end{equation} 
Here $t\in\mathbb{R}$ denotes the temporal variable, $x\in\mathbb{R}$ or $\mathbb{T}=\mathbb{R}/2\pi\mathbb{Z}$ is the spatial variable, $u=u(x,t)$ is real valued, and $\alpha \geq -1$; $\partial$ means partial differentiation and $|\partial_x|=\sqrt{-\partial_x^2}$ is a Fourier multiplier operator, defined via its symbol as
\[
\widehat{|\partial_x|f}(\xi)=|\xi|\widehat{f}(\xi).
\]
%where the circumflex means the Fourier transform. 

In the case of $\alpha=2$, \eqref{E:main} is the well-known Korteweg-de Vries equation, which was put forward in \cite{Boussinesq} and \cite{KdV} to model surface water waves of small amplitude and long wavelength in the finite depth. In the case of $\alpha=1$, \eqref{E:main} is the Benjamin-Ono equation (see \cite{Benjamin-BO}, for instance), and in the case of $\alpha=0$, it is the inviscid Burgers equation. Moreover, in the case of $\alpha=-1/2$, the author \cite{Hur-blowup} observed that \eqref{E:main} shares the dispersion relation and scaling symmetry in common with water waves in the infinite depth. Last but not least, in the case of $\alpha=-1$, \eqref{E:main} was proposed in \cite{BH} to model nonlinear waves whose linearized frequency is nonzero but constant.  

Furthermore, \eqref{E:main} belongs to the family of nonlinear dispersive equations of the form (see \cite{Whitham}, for instance)
\begin{equation}\label{E:whitham}
\partial_tu+\mathscr{M}\partial_xu+u\partial_xu=0,
\end{equation}
where $\mathscr{M}$ is a Fourier multiplier operator, defined via its symbol $m$, say. Here we assume that $m$ is real valued. Note that \eqref{E:whitham} is {\em nonlocal} unless $m$ is a polynomial of $\xi^2$. Examples include the Benjamin-Ono equation and the intermediate long wave equation (see \cite{Joseph-ILW}, for instance), for which $m(\xi)=|\xi|$ and $\xi\coth\xi$, respectively. Whitham (see \cite{Whitham}, for instance) proposed $m(\xi)=\sqrt{\tanh\xi/\xi}$ --- namely, the phase speed for surface water waves in the finite depth --- as an improvement\footnote{
Since 
\[
\sqrt{\frac{\tanh\xi}{\xi}}=1-\frac16\xi^2+O(\xi^4)\qquad\text{for $\xi\ll1$},
\]
one may regard the Korteweg-de Vries equation (after normalization of parameters) as to approximate the dispersion of the Whitham equation and, hence, water waves for low frequencies. As a matter of fact, for physically relevant initial data, the solutions of the Korteweg-de Vries equation and the Whitham equation differ from those of the water wave problem by higher order terms during a relevant time interval. But the Korteweg-de Vries equation poorly approximates the dispersion of water waves for high frequencies.} over the Korteweg-de Vries equation for high frequencies. 

In an effort to understand the competition of dispersion and nonlinearity, it is tempting, in regard to many theoretical aspects, to shift attention from \eqref{E:main} or \eqref{E:whitham}, where the nonlinearity is fixed and the dispersion varies from equation to equation, to 
\begin{equation}\label{E:gKdV}
\partial_tu+\partial_x^3u+u^p\partial_xu=0\qquad\text{(for a suitable $p$)},
\end{equation}
where the dispersion is fixed, represented by a local operator, and the nonlinearity is of variable strength, depending on $p$. The well-posedness for \eqref{E:gKdV} is worked out nearly completely. But $p$ other than $1$ or $2$ seems unlikely in practice. 

Note that \eqref{E:main} possesses three conserved quantities 
\begin{equation}\label{E:conservation}
\int\frac12u|\partial_x|^\alpha u+\frac16u^3, \qquad \int u^2, \quad\text{and}\quad \int u,
\end{equation}
which correspond to the Hamiltonian, the momentum, and the mass, respectively. For $\alpha\geq1/3$, it follows from a Sobolev inequality that the Hamiltonian is equivalent to $\|u\|_{H^{\alpha/2}}^2$. Note that \eqref{E:main} remains invariant under 
\begin{equation}\label{E:scaling}
u(x,t)\mapsto \lambda^\alpha u(\lambda x, \lambda^{\alpha+1}t)
\end{equation}
for any $\lambda>0$, whence it is $\dot{H}^{1/2-\alpha}$ critical. Here and throughout, $H^s$ and $\dot{H}^s$ denote the inhomogeneous and homogeneous, $L^2$-based Sobolev spaces. Moreover, \eqref{E:main} remains invariant under
\begin{equation}\label{E:galilean}
u(x,t)\mapsto u(x-\omega t, t)+\omega
\end{equation}
for any $\omega\in\mathbb{R}$. 

%%%%%%%%%%%%%%%%%%%%%%%%%%%%%%%%%%%%%%%%%%%%%%%%%%
%%%%%%%%%%%%%%%%%%%%%%%%%%%%%%%%%%%%%%%%%%%%%%%%%%
%%%%%%%%%%%%%%%%%%%%%%%%%%%%%%%%%%%%%%%%%%%%%%%%%%
\subsection*{Notation}
%%%%%%%%%%%%%%%%%%%%%%%%%%%%%%%%%%%%%%%%%%%%%%%%%%
%%%%%%%%%%%%%%%%%%%%%%%%%%%%%%%%%%%%%%%%%%%%%%%%%%
%%%%%%%%%%%%%%%%%%%%%%%%%%%%%%%%%%%%%%%%%%%%%%%%%%

We use $C\gg 1$ and $0<c\ll 1$ for various large and small constants, which may vary from line to line. We use $A\lesssim B$ or $B\gtrsim A$ to mean that $A\leq CB$ for some constant $C>0$, and $A\sim B$ to mean that $A\lesssim B\lesssim A$. 

%%%%%%%%%%%%%%%%%%%%%%%%%%%%%%%%%%%%%%%%%%%%%%%%%%
%%%%%%%%%%%%%%%%%%%%%%%%%%%%%%%%%%%%%%%%%%%%%%%%%%
%%%%%%%%%%%%%%%%%%%%%%%%%%%%%%%%%%%%%%%%%%%%%%%%%%
\subsection*{Earlier results}
%%%%%%%%%%%%%%%%%%%%%%%%%%%%%%%%%%%%%%%%%%%%%%%%%%
%%%%%%%%%%%%%%%%%%%%%%%%%%%%%%%%%%%%%%%%%%%%%%%%%%
%%%%%%%%%%%%%%%%%%%%%%%%%%%%%%%%%%%%%%%%%%%%%%%%%%

We say that \eqref{E:main}-\eqref{def:u0} is locally {\em well-posed} in $H^s(\mathbb{X})$, where $\mathbb{X}=\mathbb{R}$ or $\mathbb{T}$, if for any $u_0\in H^s(\mathbb{X})$, a solution of \eqref{E:main}-\eqref{def:u0} exists in $C([-t_*,t_*]; H^s(\mathbb{X}))$ for some $t_*>0$ (in the sense of distributions), it is unique in a space continuously embedded in $C([-t_*,t_*]; H^s(\mathbb{X}))$, and the map that takes an initial datum to the solution is continuous from $H^s(\mathbb{X})$ to $C([-t_*,t_*]; H^s(\mathbb{X}))$. We say that it is {\em ill-posed} otherwise, and globally well-posed if $t_*=+\infty$.

For any $\alpha\geq-1$, it follows from an a priori bound and a compactness argument that \eqref{E:main}-\eqref{def:u0} is locally well-posed in $H^s(\mathbb{X})$, $\mathbb{X}=\mathbb{R}$ or $\mathbb{T}$, provided that $s>3/2$, and $u\in C([-t_*,t_*]; H^s(\mathbb{X}))$ for some $t_*>0$; see \cite{Kato}, for instance, for details. Moreover, $t_*\gtrsim \|u_0\|_{H^s(\mathbb{X})}^{-1}$. If $u_0\in H^{s'}(\mathbb{X})$ for some $s'>s$, in addition, then $u\in C([-t_*,t_*];H^{s'}(\mathbb{X}))$. But the proof in \cite{Kato} does not improve the smoothness of the datum-to-solution map. As a matter of fact, in the case of $\alpha=0$ --- namely, the inviscid Burgers equation --- the datum-to-solution map is {\em not} uniformly continuous in $H^s(\mathbb{R})$ for any $s>3/2$; see \cite{Tzvetkov2004}, for instance, for details.

In the case of $\alpha=2$ --- namely, the Korteweg-de Vries equation --- it follows from techniques in nonlinear dispersive equations and a fixed point argument that \eqref{E:main}-\eqref{def:u0} is globally well-posed in $H^s(\mathbb{R})$ for $s\geq -3/4$ and in $H^s(\mathbb{T})$ for $s\geq -1/2$ (see \cite{CCKTT2003}, for instance). Furthermore, the datum-to-solution map is real analytic. But Christ, Colliander, and Tao \cite{CCT2003} observed that the datum-to-solution map would fail to be uniformly continuous in $H^s(\mathbb{R})$ for $-1\leq s<-3/4$ and in $H^s(\mathbb{T})$ for $-2<s<-1/2$. 

For $1\leq\alpha<2$, \eqref{E:main}-\eqref{def:u0} is globally well-posed in $L^2(\mathbb{R})$ (see \cite{HIKK2010}, for instance) and $H^{\alpha/2}(\mathbb{X})$, $\mathbb{X}=\mathbb{R}$ or $\mathbb{T}$ (see \cite{MV2015}, for instance). By the way, the $H^{\alpha/2}(\mathbb{X})$ norm is equivalent to the Hamiltonian (see \eqref{E:conservation}). %In the case of $\alpha=1$ --- namely, the Benjamin-Ono equation --- it is globally well-posed in $H^0(\mathbb{T})$ (see \cite{Molinet2008}, for instance). 
For $0<\alpha<1$, \eqref{E:main}-\eqref{def:u0} is locally well-posed in $H^s(\mathbb{R})$ for $s>3/2-3\alpha/8$ (see \cite{LPS2014}, for instance). But the proofs rely on a compactness argument, whence they may not improve the smoothness of the datum-to-solution map. 

As a matter of fact, for $0\leq\alpha<2$, Molinet, Saut, and Tzvetkov \cite{MST2001} studied interactions of low and high frequency modes, and they proved that the datum-to-solution map for \eqref{E:main}-\eqref{def:u0} would fail to be twice continuously differentiable in $H^s(\mathbb{R})$ for any $s\in\mathbb{R}$. %In particular, one would not be able to solve \eqref{E:main}-\eqref{def:u0} by a fixed point argument. 
In the case of $\alpha=2$, the same result holds for $s<-3/4$ (see \cite{Tzvetkov1999}, for instance). 

Furthermore, in the case of $\alpha=1$, Koch and Tzvetkov \cite{KT2005} exploited \eqref{E:galilean} to construct approximate solutions, and they proved that the datum-to-solution map for \eqref{E:main}-\eqref{def:u0} would fail to be uniformly continuous in $H^s(\mathbb{R})$ whenever $s>0$. For $\alpha\geq0$, Molinet \cite{Molinet2007} observed that the same result would hold in $H^s(\mathbb{T})$ for $s>0$. The same result holds for $0\leq\alpha<2$ in the non-periodic setting, for $\alpha\geq -1$ in the periodic setting, and for \eqref{E:whitham} for a broad range of the dispersion symbol (see \cite{Arnesen}, for instance). 

Moreover, for $1/3\leq\alpha\leq 1/2$, one may manipulate solitary waves to argue that the datum-to-solution map for \eqref{E:main}-\eqref{def:u0} is not uniformly continuous in $H^{1/2-\alpha}(\mathbb{R})$; see \cite{LPS2014}, for instance, for details. By the way, \eqref{E:main} is $\dot{H}^{1/2-\alpha}$ critical.

For $\alpha<0$, the well-posedness of \eqref{E:main}-\eqref{def:u0} or, rather, the lack thereof seems not adequately understood, which is the subject of investigation here. Nevertheless, for $-1<\alpha<0$, the author \cite{Hur-blowup} (see also \cite{CCG2010}) established finite time blowup in $C^{1+\gamma}(\mathbb{R})$, $0<\gamma<1$. %This is a stronger form of ill-posedness than failure of continuity of the datum-to-solution map. 
For $-1<\alpha<-1/3$, she \cite{HT2014, Hur-breaking} promoted the result to wave breaking. %That is, a solution remains bounded but its derivative becomes unbounded in finite time. 
Specifically, if $-\inf_{x\in\mathbb{R}}u_0'(x)$ is sufficiently large (and $u_0$ satisfies some technical assumptions) then the solution of \eqref{E:main}-\eqref{def:u0} exhibits that
\[|u(x,t)|<\infty\qquad\text{for all $x\in\mathbb{R}$}\quad\text{for all $t\in [0,t_*)$}\]
but
\[\inf_{x\in\mathbb{R}}\partial_xu(x,t)\to -\infty\qquad\text{as $t\to t_*-$}\]
for some $t_*>0$. Moreover, 
\begin{equation}\label{E:breaking time}
-\frac{1}{1+\epsilon}\frac{1}{\inf_{x\in\mathbb{R}}u_0'(x)}<t_*<-\frac{1}{(1-\epsilon)^2}\frac{1}{\inf_{x\in\mathbb{R}}u_0'(x)}
\end{equation}
for $\epsilon>0$ sufficiently small.
%In particular, the datum-to-solution map is {\em not} continuous in $H^{3/2}(\mathbb{R})$; see \cite{LPS2014}, for instance, for details. 

%%%%%%%%%%%%%%%%%%%%%%%%%%%%%%%%%%%%%%%%%%%%%%%%%%
%%%%%%%%%%%%%%%%%%%%%%%%%%%%%%%%%%%%%%%%%%%%%%%%%%
%%%%%%%%%%%%%%%%%%%%%%%%%%%%%%%%%%%%%%%%%%%%%%%%%%
\subsection*{Main results}
%%%%%%%%%%%%%%%%%%%%%%%%%%%%%%%%%%%%%%%%%%%%%%%%%%
%%%%%%%%%%%%%%%%%%%%%%%%%%%%%%%%%%%%%%%%%%%%%%%%%%
%%%%%%%%%%%%%%%%%%%%%%%%%%%%%%%%%%%%%%%%%%%%%%%%%%

Here we take matters further and demonstrate the {\em norm inflation} for \eqref{E:main}-\eqref{def:u0} in the periodic and non-periodic settings. Specifically, we show that an initial datum is smooth and arbitrarily small in $H^s(\mathbb{X})$, where $\mathbb{X}=\mathbb{R}$ or $\mathbb{T}$, but the solution of \eqref{E:main}-\eqref{def:u0} becomes arbitrarily large in $H^s(\mathbb{X})$ after an arbitrarily short time. This is a more drastic form of ill-posedness than the failure of uniform continuity of the datum-to-solution map, and it implies that the datum-to-solution map for \eqref{E:main}-\eqref{def:u0} is discontinuous at the origin in the $H^s(\mathbb{X})$ topology. 

%%%%%%%%%%%%%%%%%%%%%%%%%%%%%%%%%%%%%%%%%%%%%%%%%%
%%%%%%%%%%%%%%%%%%%%%%%%%%%%%%%%%%%%%%%%%%%%%%%%%%
%%%%%%%%%%%%%%%%%%%%%%%%%%%%%%%%%%%%%%%%%%%%%%%%%%
\begin{theorem}[Norm inflation in $\mathbb{R}$]\label{thm:R}
Let $-1\leq \alpha<-1/3$, and assume that $5/6<s<1/2-\alpha$. For any $\epsilon>0$, there exist $u_0$ in the Schwartz class, $t$ in the interval $(0,\epsilon)$, and the solution $u$ of \eqref{E:main}-\eqref{def:u0} such that 
\[
\|u_0\|_{H^s(\mathbb{R})}<\epsilon \quad\text{but}\quad \|u(\cdot,t)\|_{H^s(\mathbb{R})}>\epsilon^{-1}.
\]
\end{theorem}
%%%%%%%%%%%%%%%%%%%%%%%%%%%%%%%%%%%%%%%%%%%%%%%%%%
%%%%%%%%%%%%%%%%%%%%%%%%%%%%%%%%%%%%%%%%%%%%%%%%%%
%%%%%%%%%%%%%%%%%%%%%%%%%%%%%%%%%%%%%%%%%%%%%%%%%%

Theorem~\ref{thm:R} implies that the datum-to-solution map for \eqref{E:main}-\eqref{def:u0}, which exists from $H^s(\mathbb{R})$ to $C([-t_*,t_*]; H^s(\mathbb{R}))$ for some $t_*>0$ when $s>3/2$, may not be continuously extended to $5/6<s<1/2-\alpha\,(<3/2)$. In particular, \eqref{E:main}-\eqref{def:u0} is ill-posed in $H^s(\mathbb{R})$ for $5/6<s<1/2-\alpha$. 

Recall that \eqref{E:main} is $\dot{H}^{1/2-\alpha}(\mathbb{R})$ critical, whence in Theorem~\ref{thm:R}, the norm inflation takes place in a range of supercritical Sobolev spaces. Note that $5/6=1/2-\alpha$ when $\alpha=-1/3$. The restriction $s<5/6$ may be an artifact of the method of the proof. Perhaps, a better understanding of the blowup of solutions of the inviscid Burgers equation improves this. Note that $1/2-\alpha=3/2$ when $\alpha=-1$. Thus the local well-posedness result of \eqref{E:main}-\eqref{def:u0} in $H^s(\mathbb{R})$ for $s>3/2$ is sharp when $\alpha=-1$. 

The proof of Theorem~\ref{thm:R} is similar to that in \cite{CCT-R} for nonlinear Schr\"odinger equations, combining scaling symmetry (see \eqref{E:scaling}) and a quantitative study of the equation in the {\em zero dispersion limit} (see \eqref{E:v} and \eqref{E:Burgers}). But the main difference lies in that the inviscid Burgers equation in the zero dispersion limit may be solved exactly, but implicitly, and its solution blows up in finite time. 

In the usual well-posedness theory, one would regard \eqref{E:main} as a perturbation of the linear equation. Here we take the opposite point of view and regard the equation as a dispersive perturbation of the inviscid Burgers equation. We show that the solution of \eqref{E:v} and \eqref{def:un} remains close to the solution of the inviscid Burgers equation, which by the way blows up in finite time, for small values of the dispersion parameter. We then vary the scaling and dispersion parameters so that the initial datum is sufficiently small in the desired Sobolev space but the solution of \eqref{E:main}-\eqref{def:u0} becomes sufficiently large in the Sobolev space after a sufficiently short time. %We study in detail the asymptotic behavior of the solution of the inviscid Burgers equation near blowup, and we control the solutions of \eqref{E:main}-\eqref{def:u0} uniformly for the dispersion parameter.

The present treatment may be adapted to a broad class of nonlinear dispersive equations, provided that they enjoy scaling symmetry and the solutions in the zero dispersion limit grow unboundedly large in finite or infinite time, for instance, to the water wave problem. This is an interesting direction of future research.

%%%%%%%%%%%%%%%%%%%%%%%%%%%%%%%%%%%%%%%%%%%%%%%%%%
%%%%%%%%%%%%%%%%%%%%%%%%%%%%%%%%%%%%%%%%%%%%%%%%%%
%%%%%%%%%%%%%%%%%%%%%%%%%%%%%%%%%%%%%%%%%%%%%%%%%%
\begin{theorem}[Norm inflation in $\mathbb{T}$]\label{thm:T}
Let $-1\leq\alpha<2$, and assume that $s<-2$. For every $\epsilon>0$, there exist $u_0\in C^\infty(\mathbb{T})$, $t$ in the interval $(0,\epsilon)$, and the solution $u$ of \eqref{E:main}-\eqref{def:u0} such that 
\[
\|u_0\|_{\dot{H}^s(\mathbb{T})}<\epsilon \quad\text{but}\quad \|u(\cdot,t)\|_{\dot{H}^s(\mathbb{T})}>\epsilon^{-1}.
\] 
\end{theorem}
%%%%%%%%%%%%%%%%%%%%%%%%%%%%%%%%%%%%%%%%%%%%%%%%%%
%%%%%%%%%%%%%%%%%%%%%%%%%%%%%%%%%%%%%%%%%%%%%%%%%%
%%%%%%%%%%%%%%%%%%%%%%%%%%%%%%%%%%%%%%%%%%%%%%%%%%

Theorem~\ref{thm:T} implies that the datum-to-solution map for \eqref{E:main}-\eqref{def:u0}, were it to exist in $H^s(\mathbb{T})$ for all $s\in\mathbb{R}$, is discontinuous at the origin for $s<-2$ (But it is continuous for any $s>3/2$). Similar results hold in the non-periodic setting and for \eqref{E:whitham} for a broad range of the dispersion symbol, but we do not include the details here. 

The proof of Theorem~\ref{thm:T} is to construct an explicit approximate solution which enjoys the desired norm inflation behavior, and then to use an a priori bound to show that the solution remains close to the approximate solution. 

Perhaps, the simplest type of initial datum in the periodic setting is $\cos(nx)$ for $n\in\mathbb{N}$. Note that $\cos(nx+n^{\alpha+1}t)$ solves the linear part of \eqref{E:main} and $u(x,0)=\cos(nx)$. It then follows from \eqref{E:galilean} that 
\[
\cos(nx+n^{\alpha+1}t-\omega t)+\omega
\]
solves the linear part of \eqref{E:main} and $u(x,0)=\cos(nx)+\omega$ for any $\omega\in\mathbb{R}$. Molinet \cite{Molinet2007} used this to prove the failure of uniform continuity of the datum-to-solution map for \eqref{E:main}-\eqref{def:u0} in $H^s(\mathbb{T})$ for any $s>0$. But the datum-to-solution map for the Benjamin-Ono equation is uniformly continuous once we restrict the attention to functions of fixed mean value, so that \eqref{E:galilean} may not apply. Here we work with functions of mean zero, which prevents us from manipulating \eqref{E:galilean}, and, instead, we develop the next simplest type of initial datum, supported on two adjacent, high frequency modes (see \eqref{def:u0T}). We then show that the nonlinear interaction of the high frequency modes drives oscillation of a low frequency mode, which is larger in Sobolev spaces of negative exponents. 

%%%%%%%%%%%%%%%%%%%%%%%%%%%%%%%%%%%%%%%%%%%%%%%%%%
%%%%%%%%%%%%%%%%%%%%%%%%%%%%%%%%%%%%%%%%%%%%%%%%%%
%%%%%%%%%%%%%%%%%%%%%%%%%%%%%%%%%%%%%%%%%%%%%%%%%%
\section{Proof of Theorem~\ref{thm:R}}\label{sec:R}
%%%%%%%%%%%%%%%%%%%%%%%%%%%%%%%%%%%%%%%%%%%%%%%%%%
%%%%%%%%%%%%%%%%%%%%%%%%%%%%%%%%%%%%%%%%%%%%%%%%%%
%%%%%%%%%%%%%%%%%%%%%%%%%%%%%%%%%%%%%%%%%%%%%%%%%%

Let $-1\leq\alpha<-1/3$, and assume that $5/6<s<1/2-\alpha$. Assume that $u_0$ is nonzero Schwartz function. 

For $\nu>0$, we relate \eqref{E:main}-\eqref{def:u0} to 
\begin{equation}\label{E:v}
\partial_Tu^{(\nu)}+\nu^{-\alpha}|\partial_X|^\alpha \partial_Xu^{(\nu)}+u^{(\nu)}\partial_Xu^{(\nu)}=0
\end{equation}
and 
\begin{equation}\label{def:un}
u^{(\nu)}(X,0)=u_0(X),
\end{equation} 
via 
\begin{equation}\label{def:v}
u(x,t)=u^{(\nu)}(x/\nu, t/\nu).
\end{equation}
%In other words, $X=x/\nu$ and $T=t/\nu$. 
As $\nu\to0$, formally, \eqref{E:v} tends to the inviscid Burgers equation
\begin{equation}\label{E:Burgers}
\partial_Tu^{(0)}+u^{(0)}\partial_Xu^{(0)}=0.
\end{equation}

%Recall from \cite{Kato}, for instance, that a unique solution of \eqref{E:main}-\eqref{def:u0} exists in $C([0,t_*); H^s(\mathbb{R}))$ for some $t_*>0$ for any $s>3/2$. Assume that $t_*$ is the maximal time of existence. Recall from \cite{Hur-breaking} that the solution may blow up only as a result of wave breaking at the time satisfying \eqref{E:breaking time}. Therefore, $t_*>-\frac{1}{\inf_{x\in\mathbb{R}}u_0'(x)}$.

Recall from the well-posedness theory (see \cite{Kato}, for instance) that for any $\nu\geq0$, a unique solution of \eqref{E:v} (or \eqref{E:Burgers}) and \eqref{def:un} exists in $C((-T_*^{(\nu)},T_*^{(\nu)})); H^{s_*}(\mathbb{R}))$ for some $T_*^{(\nu)}>0$, provided that $s_*>3/2$. Let $T_*^{(\nu)}$ be the maximal time of existence. When $\nu>0$, recall from \cite{Hur-breaking} that the solution of \eqref{E:v}-\eqref{def:un} blows up merely as a result of wave breaking at the time satisfying \eqref{E:breaking time}. Therefore, 
\[
T_*^{(\nu)}>\frac{1}{-\inf_{x\in\mathbb{R}}u_0'(x))}-0^+
\]
independently of $\nu>0$.

When $\nu=0$, for $x\in\mathbb{R}$ (by abuse of notation), let $X(T;x)$ solve
\begin{equation}\label{def:char}
\frac{dX}{dT}(T;x)=u^{(0)}(X(T;x),T)\quad\text{and}\quad X(0;x)=x.
\end{equation}
Since $u^{(0)}(X,T)$ is bounded and satisfies a Lipschitz condition in $X$ for all $X\in\mathbb{R}$ for all $T\in(-T_*^{(0)},T_*^{(0)})$, it follows from the theory of ordinary differential equations that $X(\cdot;x)$ exists throughout the interval $(-T_*^{(0)},T_*^{(0)})$ for all $x\in\mathbb{R}$. Furthermore, $x\mapsto X(\cdot;x)$ is continuously differentiable throughout the interval $(-T_*^{(0)},T_*^{(0)})$ for all $x\in\mathbb{R}$.

It is well known that one may solve \eqref{E:Burgers} and \eqref{def:un} by the method of characteristics. Specifically,
\begin{equation}\label{E:Burgers soln}
u^{(0)}(X(T;x),T)=u^{(0)}(X(0;x),0)=u_0(x).
\end{equation}
Differentiating \eqref{def:char} with respect to $x$, we use \eqref{E:Burgers soln} to arrive at
\[
\frac{d\partial_xX}{dT}(T;x)=u_0'(x)\quad\text{and}\quad \partial_xX(0;x)=1,
\]
whence
\begin{equation}\label{E:X_x}
\partial_xX(T;x)=1+u_0'(x)T.
\end{equation}
Note that if $u_0'(x)<0$ for some $x\in\mathbb{R}$ then 
\[
\partial_Xu^{(0)}(X(T;x),T)=\frac{u_0'(x)}{1+u_0'(x)T}
\]
becomes unbounded at such $x$ in finite time. Therefore, 
\begin{equation}\label{E:t*}
T_*^{(0)}=-\frac{1}{\inf_{x\in\mathbb{R}}u_0'(x)}=:-\frac{1}{u_0'(x_*)}
\end{equation}
for some $x_*\in\mathbb{R}$. In what follows, we write $T_*$ for $T_*^{(0)}$ for simplicity of notation.

A straightforward calculation reveals that
\begin{align*}
\partial_Xu^{(0)}(X,T)=&\frac{u_0'(x)}{1+u_0'(x)T}\leq\frac{\|u_0'\|_{C^0(\mathbb{R})}}{1+u_0'(x_*)T},
\intertext{and}
\partial_X^2u^{(0)}(X,T)=&\frac{u_0''(x)}{(1+u_0'(x)T)^3}\leq \frac{\|u_0''\|_{C^0(\mathbb{R})}}{(1+u_0'(x_*)T)^3}
\end{align*}
pointwise in $\mathbb{R}$ for any $0<T<T_*$. Therefore,
\begin{equation}\label{E:v0Hk}
\|u^{(0)}(\cdot,T)\|_{H^k(\mathbb{R})}, \|u^{(0)}(\cdot,T)\|_{C^k(\mathbb{R})}\leq \frac{C}{(1+u_0'(x_*)T)^{k+1}}
\end{equation}
for any $0<T<T_*$ for any integer $k\geq 1$. 
Below we study the asymptotic behavior of the solution of \eqref{E:Burgers} and \eqref{def:un} near blowup, and we compute $\|u^{(0)}(\cdot,T)\|_{H^s(\mathbb{R})}$ for $s<1$ for $T$ close to $T_*$. 

%%%%%%%%%%%%%%%%%%%%%%%%%%%%%%%%%%%%%%%%%%%%%%%%%%
%%%%%%%%%%%%%%%%%%%%%%%%%%%%%%%%%%%%%%%%%%%%%%%%%%
%%%%%%%%%%%%%%%%%%%%%%%%%%%%%%%%%%%%%%%%%%%%%%%%%%
\begin{lemma}\label{lem:Burgers asym}
Assume that $u_0$ is a nonzero Schwartz function. If the solution of \eqref{E:Burgers} and \eqref{def:un} blows up at some $X_*\in\mathbb{R}$ and at $T_*=T_*^{(0)}>0$ then 
\begin{equation}\label{E:Burgers asym}
u^{(0)}(X,T)\sim u^{(0)}(X_*,T_*)-\frac{1}{T_*}(T_*-T)^{1/2}U\Big(\frac{X-X_*}{(T_*-T)^{3/2}}\Big)+o((T_*-T)^{1/2})
\end{equation}
as $T\to T_*-$ uniformly for $|X-X_*|\lesssim (T_*-T)^{3/2}$, where $U=U(Y)$ is real valued and satisfies 
\begin{equation}\label{E:U}
C_1U(Y)+C_3U^3(Y)=Y
\end{equation}
for some constants $C_1,C_3>0$, and $o((T_*-T)^{1/2})$ is a function of $\frac{X-X_*}{(T_*-T)^{3/2}}$. 
\end{lemma}
%%%%%%%%%%%%%%%%%%%%%%%%%%%%%%%%%%%%%%%%%%%%%%%%%%
%%%%%%%%%%%%%%%%%%%%%%%%%%%%%%%%%%%%%%%%%%%%%%%%%%
%%%%%%%%%%%%%%%%%%%%%%%%%%%%%%%%%%%%%%%%%%%%%%%%%%

%%%%%%%%%%%%%%%%%%%%%%%%%%%%%%%%%%%%%%%%%%%%%%%%%%
%%%%%%%%%%%%%%%%%%%%%%%%%%%%%%%%%%%%%%%%%%%%%%%%%%
%%%%%%%%%%%%%%%%%%%%%%%%%%%%%%%%%%%%%%%%%%%%%%%%%%
\begin{proof}
Without loss of generality, we may assume that $u_0(x_*)=0$. As a matter of fact, \eqref{E:Burgers} remains invariant under $X\mapsto X-u_0(x_*)T$ and $u^{(0)}\mapsto u^{(0)}+u_0(x_*)$. Therefore, \eqref{E:Burgers soln} implies $X_*=x_*$. Moreover, \eqref{E:t*} implies
\begin{equation}\label{E:u(x*)}
u_0'(x_*)=-\frac{1}{T_*}<0, \quad u_0''(x_*)=0,\quad\text{and}\quad u_0'''(x_*)>0.
\end{equation}

For $|x-x_*|$ and $|T-T_*|$ sufficiently small, we expand \eqref{E:X_x} and we use $u_0(x_*)=0$ and \eqref{E:u(x*)} to arrive at
\[
\partial_xX(T;x)=u_0'(x_*)(T-T_*)+\frac12u_0'''(x_*)T_*(x-x_*)^2+o((T-T_*)^2+(x-x_*)^2).
\]
An integration then leads to 
\begin{align*}
X(T;x)-X(T;x_*)=u_0'(x_*)(T-T_*)(x-x_*)+&\frac16u_0'''(x_*)T_*(x-x_*)^3\\ +&o(((T-T_*)^2+(x-x_*)^2)(x-x_*))
\end{align*}
as $x\to x_*$ and $T\to T_*-$. Note that $X(T;x_*)=X_*=x_*$ for any $0\leq T<T_*$. Therefore, we use \eqref{E:t*} to deduce that
\begin{align}\label{E:X-X*}
X(T;x)-X_*=u_0'(x_*)(T-T_*)(x-x_*)-&\frac16\frac{u_0'''(x_*)}{u_0'(x_*)}(x-x_*)^3\\ 
+&o(((T-T_*)^2+(x-x_*)^2)(x-x_*))\notag
\end{align}
as $x\to x_*$ and $T\to T_*-$.
Moreover, we expand \eqref{E:Burgers soln} to arrive at
\begin{equation}\label{E:Burgers soln'}
u^{(0)}(X(T;x),T)=u_0'(x_*)(x-x_*)+o((x-x_*))
\end{equation}
as $x\to x_*$. 

Let
\begin{equation}\label{def:self similar}
Y=\frac{X-X_*}{(T_*-T)^{3/2}}\quad\text{and}\quad U=\frac{x-x_*}{(T_*-T)^{1/2}}.
\end{equation}
By the way, this is a similarity solution.
For $|X-X_*|$ and $|T-T_*|$ sufficiently small, satisfying $|Y|\lesssim 1$, a straightforward calculation reveals that \eqref{E:X-X*} becomes
\[
Y=-u_0'(x_*)U-\frac16\frac{u_0'''(x_*)}{u_0'(x_*)}U^3+o(U+U^3).
\]
Note that $Y=-u_0'(x_*)U-\frac16\frac{u_0'''(x_*)}{u_0'(x_*)}U^3$ supports a unique and real-valued solution $U=U(Y)$, say. For $|X-X_*|$ and $|T-T_*|$ sufficiently small, satisfying $|Y|\lesssim 1$, similarly, \eqref{E:Burgers soln'} becomes
\[
u^{(0)}(X,T)=u_0'(x_*)(T_*-T)^{1/2}U(Y)+o((T_*-T)^{1/2}),
\]
where $o((T_*-T)^{1/2})$ is a function of $Y$. This completes the proof.
\end{proof}
%%%%%%%%%%%%%%%%%%%%%%%%%%%%%%%%%%%%%%%%%%%%%%%%%%
%%%%%%%%%%%%%%%%%%%%%%%%%%%%%%%%%%%%%%%%%%%%%%%%%%
%%%%%%%%%%%%%%%%%%%%%%%%%%%%%%%%%%%%%%%%%%%%%%%%%%

%%%%%%%%%%%%%%%%%%%%%%%%%%%%%%%%%%%%%%%%%%%%%%%%%%
%%%%%%%%%%%%%%%%%%%%%%%%%%%%%%%%%%%%%%%%%%%%%%%%%%
%%%%%%%%%%%%%%%%%%%%%%%%%%%%%%%%%%%%%%%%%%%%%%%%%%
\begin{corollary}\label{cor:Burgers norm}
Under the hypothesis of Lemma~\ref{lem:Burgers asym}, for any $s>0$,
\begin{equation}\label{E:u0sR}
\|u^{(0)}(\cdot,T)\|_{\dot{H}^s(\mathbb{R})}\gtrsim (T_*-T)^{5/4-3s/2}
\end{equation}
as $T\to T_*-$. 
\end{corollary}
%%%%%%%%%%%%%%%%%%%%%%%%%%%%%%%%%%%%%%%%%%%%%%%%%%
%%%%%%%%%%%%%%%%%%%%%%%%%%%%%%%%%%%%%%%%%%%%%%%%%%
%%%%%%%%%%%%%%%%%%%%%%%%%%%%%%%%%%%%%%%%%%%%%%%%%%

In particular, when $s>5/6$, $\|u^{(0)}(\cdot,T)\|_{\dot{H}^s(\mathbb{R})} \to \infty$ as $T\to T_*-$.

%%%%%%%%%%%%%%%%%%%%%%%%%%%%%%%%%%%%%%%%%%%%%%%%%%
%%%%%%%%%%%%%%%%%%%%%%%%%%%%%%%%%%%%%%%%%%%%%%%%%%
%%%%%%%%%%%%%%%%%%%%%%%%%%%%%%%%%%%%%%%%%%%%%%%%%%
\begin{proof}
For $|T-T_*|$ sufficiently small, we calculate 
\begin{align*}
\|u^{(0)}(\cdot,&T)\|_{\dot{H}^s(\mathbb{R})}^2 \\ 
\geq& \int_{|X-X_*|\lesssim (T_*-T)^{3/2}}||\partial_X|^su^{(0)}(X,T)|^2~dX \\
\sim&(T_*-T)\int_{|X-X_*|\lesssim(T_*-T)^{3/2}}
\Big||\partial_X|^s\Big(\frac{1}{T_*}U\Big(\frac{X-X_*}{(T_*-T)^{3/2}}\Big)+o(1)\Big)\Big|^2~dX \\
\sim&(T_*-T)^{5/2-3s}\int_{|Y|\lesssim1}||\partial_Y|^sU(Y)|^2~dY+o((T_*-T)^{5/2-3s}),
\end{align*}
and \eqref{E:u0sR} follows. Here the second inequality uses \eqref{E:Burgers asym}. Note that $o(1)$ is a function of $Y=\frac{X-X_*}{(T_*-T)^{3/2}}$. The last inequality uses \eqref{def:self similar}. 
\end{proof}
%%%%%%%%%%%%%%%%%%%%%%%%%%%%%%%%%%%%%%%%%%%%%%%%%%
%%%%%%%%%%%%%%%%%%%%%%%%%%%%%%%%%%%%%%%%%%%%%%%%%%
%%%%%%%%%%%%%%%%%%%%%%%%%%%%%%%%%%%%%%%%%%%%%%%%%%

For $\nu>0$ small, one may expect that the solutions of \eqref{E:v} and \eqref{E:Burgers} subject to the same initial condition remain close to each other at least for short times. Below we make this precise for a time interval, depending on $\nu$, which tends to $(0,T_*)$ as $\nu\to 0$. 

%%%%%%%%%%%%%%%%%%%%%%%%%%%%%%%%%%%%%%%%%%%%%%%%%%
%%%%%%%%%%%%%%%%%%%%%%%%%%%%%%%%%%%%%%%%%%%%%%%%%%
%%%%%%%%%%%%%%%%%%%%%%%%%%%%%%%%%%%%%%%%%%%%%%%%%%
\begin{lemma}\label{lem:energy}
Assume that $u_0$ is a nonzero Schwartz function. Assume that $u^{(\nu)}$ solves \eqref{E:v} and \eqref{def:un}, and $u^{(0)}$ solves \eqref{E:Burgers} and \eqref{def:un} during the interval $(-T_*+0^+ ,T_*-0^+)$. For $\nu>0$ sufficiently small and $k\geq 2$ an integer, 
\begin{equation*}\label{E:energy}
\|(u^{(\nu)}-u^{(0)})(\cdot,T)\|_{H^k(\mathbb{R})} \lesssim \nu^{-\alpha/2} 
\qquad \text{for any $0<T\leq T_*\Big(1-\Big(\frac{C}{|\log\nu|}\Big)^{1/C}\Big)$}
\end{equation*}
for some constant $C>0$.
\end{lemma}
%%%%%%%%%%%%%%%%%%%%%%%%%%%%%%%%%%%%%%%%%%%%%%%%%%
%%%%%%%%%%%%%%%%%%%%%%%%%%%%%%%%%%%%%%%%%%%%%%%%%%
%%%%%%%%%%%%%%%%%%%%%%%%%%%%%%%%%%%%%%%%%%%%%%%%%%

In particular, for $\nu>0$ sufficiently small and for $s>5/6$, we combine this and Corollary~\ref{cor:Burgers norm} to deduce that 
\begin{equation}\label{E:blowup-v}
\|u^{(\nu)}(\cdot,T)\|_{\dot{H}^s(\mathbb{R})}\sim (T_*-T)^{5/4-3s/2} \to \infty \qquad \text{as}\quad T\to T_*-.
\end{equation}

%%%%%%%%%%%%%%%%%%%%%%%%%%%%%%%%%%%%%%%%%%%%%%%%%%
%%%%%%%%%%%%%%%%%%%%%%%%%%%%%%%%%%%%%%%%%%%%%%%%%%
%%%%%%%%%%%%%%%%%%%%%%%%%%%%%%%%%%%%%%%%%%%%%%%%%%
\begin{proof}
The proof uses the energy method and is rudimentary. Here we include details for the sake of completeness.

Let $w=u^{(\nu)}-u^{(0)}$. Note from \eqref{E:v} and \eqref{E:Burgers} that  
\[
\partial_Tw+\nu^{-\alpha}|\partial_X|^\alpha\partial_Xw+\nu^{-\alpha}|\partial_X|^\alpha\partial_Xu^{(0)}
+\partial_X(u^{(0)}w)+w\partial_Xw=0
\]
and $w(X,0)=0$. 
Differentiating this $j$ times with respect to $X$, where $0\leq j\leq k$ an integer, and integrating over $\mathbb{R}$ against $\partial_X^jw$, we arrive at
\begin{multline*}
\int_\mathbb{R} \partial^j_Xw\partial^j_X\partial_Tw 
+\nu^{-\alpha}\int_\mathbb{R}\partial^j_Xw|\partial_X|^\alpha\partial^{j+1}_Xw\\
+\nu^{-\alpha}\int_\mathbb{R}\partial^j_Xw|\partial_X|^\alpha\partial^{j+1}_Xu^{(0)} 
+\int_\mathbb{R}\partial^j_Xw\partial^{j+1}_X(u^{(0)}w)+\frac12\int_\mathbb{R}\partial^j_Xw\partial^{j+1}_Xw^2=0.
\end{multline*}
The second term on the left side vanishes by a symmetry argument. A straightforward calculation then reveals that 
\begin{multline*}
\frac12\frac{d}{dT}\|w(\cdot,T)\|_{H^k(\mathbb{R})}^2 \leq 
\nu^{-\alpha}\|u^{(0)}(\cdot,T)\|_{H^{k+2}(\mathbb{R})}\|w(\cdot,T)\|_{H^k(\mathbb{R})} \\
+C\|u^{(0)}(\cdot,T)\|_{H^{k+2}(\mathbb{R})}\|w(\cdot,T)\|_{H^k}^2+C\|w(\cdot,T)\|_{H^k(\mathbb{R})}^3
\end{multline*}
for any $-T_*+0^+<T<T_*-0^+$.

We assume a priori that $\|w(\cdot,T)\|_{H^k(\mathbb{R})}\leq 1$ for any $-T_*+0^+<T<T_*-0^+$. We use \eqref{E:v0Hk} to show that
\[
\frac{d}{dT}\|w(\cdot,T)\|_{H^k(\mathbb{R})}\leq 
\nu^{-\alpha}\frac{C}{(1+u_0'(x_*)T)^{k+3}}+\frac{C}{(1+u_0'(x_*)T)^{k+3}}\|w(\cdot,T)\|_{H^k(\mathbb{R})}
\]
for any $-T_*+0^+<T<T_*-0^+$. It then follows from Gronwall's lemma that 
\[
\|w(\cdot,T)\|_{H^k(\mathbb{R})}\leq C\nu^{-\alpha}\exp\Big(\frac{C}{(1+u_0'(x_*)T)^C}\Big)\leq C\nu^{-\alpha},
\]
provided that $-T_*+0^+<T<T_*-0^+$ and
\[
\nu^{-\alpha/2}\exp\Big(\frac{C}{(1+u_0'(x_*)T)^C}\Big)\leq 1,
\]
or, equivalently,
\[
T \leq T_*\Big(1-\Big(\frac{C}{|\log\nu|}\Big)^{1/C}\Big),
\]
which tends to $T_*$ as $\nu\to0$.
Note that we recover the a priori assumption $\|w(\cdot,T)\|_{H^k(\mathbb{R})}\leq 1$,
and we may remove it by the usual continuity argument. This completes the proof.
\end{proof}
%%%%%%%%%%%%%%%%%%%%%%%%%%%%%%%%%%%%%%%%%%%%%%%%%%
%%%%%%%%%%%%%%%%%%%%%%%%%%%%%%%%%%%%%%%%%%%%%%%%%%
%%%%%%%%%%%%%%%%%%%%%%%%%%%%%%%%%%%%%%%%%%%%%%%%%%

We merely pause to remark that in \cite{CCT-R} for the nonlinear Schr\"odinger equations, the Cauchy problem is globally well-posed in a certain Sobolev space, and the equation at the zero dispersion limit admits an explicit solution, which grows like $T^s$ in $H^s(\mathbb{R})$ for any $s>0$. 

\

We now use Corollary~\ref{cor:Burgers norm} and Lemma~\ref{lem:energy} to prove Theorem~\ref{thm:R}.

Let $-1\leq \alpha<1/3$. Assume that $5/6<s<1/2-\alpha=:s_c$ and $u_0$ is a nonzero but arbitrary Schwartz function. For $\lambda, \nu>0$, let 
\begin{equation}\label{def:u^ln}
u^{(\lambda, \nu)}(x,t):=\lambda^\alpha u^{(\nu)}(\lambda x/\nu,\lambda^{\alpha+1}t/\nu),
\end{equation}
where $u^{(\nu)}$ solves \eqref{E:v} and \eqref{def:un}. It is straightforward to verify that $u^{(\lambda,\nu)}$ solves \eqref{E:main} and
\[
u^{(\lambda, \nu)}(x,0)=\lambda^\alpha u_0(\lambda x/\nu).
\]
%We assume that the solution of \eqref{E:Burgers} and $u^{(0)}(X,0)=u_0(X)$ breaks down at time $T_*$. 
For $\epsilon>0$ sufficiently small, we shall show that 
\begin{equation}\label{E:claim}
\|u^{(\lambda,\nu)}(\cdot,0)\|_{H^s(\mathbb{R})}\lesssim\epsilon\quad\text{but}\quad
\|u^{(\lambda,\nu)}(\cdot,\nu T/\lambda^{\alpha+1})\|_{H^s(\mathbb{R})}\gtrsim\epsilon^{-1}
\end{equation}
for some $0<\nu\leq \lambda \ll 1$ and $T\sim T_*$. 

\

We begin by calculating 
\begin{align*}
\|u^{(\lambda,\nu)}(\cdot,0)\|_{H^s(\mathbb{R})}^2
=&\lambda^{2\alpha}(\nu/\lambda)^2\int_\mathbb{R}(1+|\xi|^2)^s|\widehat{u_0}(\nu\xi/\lambda)|^2~d\xi \\
=&\lambda^{2\alpha}(\nu/\lambda)\int_\mathbb{R}(1+|\lambda\eta/\nu|^2)^s|\widehat{u_0}(\eta)|^2~d\eta \\
\sim&\lambda^{2\alpha}(\nu/\lambda)^{1-2s}
\int_{|\eta|\geq \nu/\lambda}|\eta|^{2s}|\widehat{u_0}(\eta)|^2~d\eta
+\lambda^{2\alpha}(\nu/\lambda)\int_{|\eta|\leq \nu/\lambda}|\widehat{u_0}(\eta)|^2~d\eta \\
=&\lambda^{2\alpha}(\nu/\lambda)^{1-2s}\int_\mathbb{R}|\eta|^{2s}|\widehat{u_0}(\eta)|^2~d\eta \\
&-\lambda^{2\alpha}(\nu/\lambda)^{1-2s}\int_{|\eta|\leq\nu/\lambda}
|\widehat{u_0}(\eta)|^2((\nu/\lambda)^{2s}-|\eta|^{2s})~d\eta\\
=&c\lambda^{2\alpha}(\nu/\lambda)^{1-2s}(1+O(\nu/\lambda)^{1+1/2})
\end{align*}
for some constant $c>0$. Therefore, for $0<\nu\leq \lambda$,
\[
\|u^{(\lambda,\nu)}(\cdot,0)\|_{H^s(\mathbb{R})}\leq C\lambda^\alpha(\nu/\lambda)^{1/2-s}
=C\lambda^{s-s_c}\nu^{1/2-s}.
\]
Let 
\[
\lambda^{s-s_c}\nu^{1/2-s}=\epsilon,
\] 
or, equivalently, 
\[
\nu=c\lambda^{(s_c-s)/(1/2-s)}
\] 
for some $c>0$. Note that $(s_c-s)/(1/2-s)>1$. Therefore, $0\leq\nu\leq\lambda$ as $\lambda\to 0$. This proves the former inequality of \eqref{E:claim}.

To proceed, we calculate 
\begin{align*}
\|u^{(\lambda,\nu)}(\cdot,\nu T/\lambda^{\alpha+1})\|_{H^s(\mathbb{R})}^2
=&\lambda^{2\alpha}(\nu/\lambda)^2\int_\mathbb{R}(1+|\xi|^2)^{2s}|(\widehat{u^{(\nu)}(\cdot,T)})(\nu\xi/\lambda)|^2~d\xi \\
=&\lambda^{2\alpha}(\nu/\lambda)
\int_\mathbb{R}(1+|\lambda\eta/\nu|^2)^{2s}|(\widehat{u^{(\nu)}(\cdot,T)})(\eta)|^2~d\eta \\
\geq&\lambda^{2\alpha}(\nu/\lambda)^{1-2s}
\int_{|\eta|\geq1}|\eta|^{2s}|(\widehat{u^{(\nu)}(\cdot,T)})(\eta)|^2~d\eta\\
\geq&\lambda^{2\alpha}(\nu/\lambda)^{1-2s}
(c\|u^{(\nu)}(\cdot,T)\|_{H^s(\mathbb{R})}^2-C\|u^{(\nu)}(\cdot,T)\|_{L^2(\mathbb{R})}^2).
\end{align*}
Here the first equality uses \eqref{def:u^ln} and the last inequality uses the definition of $H^s(\mathbb{R})$ and $\dot{H}^s(\mathbb{R})$.

Note from \eqref{E:conservation} that 
\[
\|u^{(\nu)}(\cdot,T)\|_{H^s(\mathbb{R})}\geq \|u^{(\nu)}(\cdot,T)\|_{L^2(\mathbb{R})}=\|u^{(\nu)}(\cdot,0)\|_{L^2(\mathbb{R})}
\]
for any $0<T<T_*-0^+$. On the other hand,  \eqref{E:blowup-v} implies 
\[
\|u^{(\nu)}(\cdot, T)\|_{H^s(\mathbb{R})}\gtrsim (T_*-T)^{5/4-3s/2}
\] 
as $T\to T_*-$. Therefore,
\begin{align*}
\|u^{(\lambda,\nu)}(\cdot,\nu T/\lambda^{\alpha+1})\|_{H^s(\mathbb{R})}
\geq &\lambda^\alpha(\nu/\lambda)^{1/2-s}\|v^{(\nu)}(\cdot,T)\|_{H^s(\mathbb{R})}\\
\gtrsim &c\epsilon (T-T_*)^{5/4-3s/2}
\end{align*}
as $T\to T_*-$.
The latter inequality of \eqref{E:blowup-v} then follows 
upon choosing $T$ sufficiently close to $T_*$ depending on $\epsilon$, 
and choosing $\nu$ and, hence, $\lambda$ sufficiently small depending on $\epsilon$ and $T$, so that $\nu T/\lambda^{\alpha+1}<\epsilon$ and $(T-T_*)^{5/4-3s/2}>\epsilon^{-2}$. This completes the proof.

%%%%%%%%%%%%%%%%%%%%%%%%%%%%%%%%%%%%%%%%%%%%%%%%%%
%%%%%%%%%%%%%%%%%%%%%%%%%%%%%%%%%%%%%%%%%%%%%%%%%%
%%%%%%%%%%%%%%%%%%%%%%%%%%%%%%%%%%%%%%%%%%%%%%%%%%
\section{Proof of Theorem~\ref{thm:T}}\label{sec:T}
%%%%%%%%%%%%%%%%%%%%%%%%%%%%%%%%%%%%%%%%%%%%%%%%%%
%%%%%%%%%%%%%%%%%%%%%%%%%%%%%%%%%%%%%%%%%%%%%%%%%%
%%%%%%%%%%%%%%%%%%%%%%%%%%%%%%%%%%%%%%%%%%%%%%%%%%

Let $-1\leq\alpha<2$, and assume that $s<-2$. For $\epsilon>0$ sufficiently small and for $n\in\mathbb{N}$ sufficiently large, to be determined in the course of the proof, let
\begin{equation}\label{def:u0T}
u_0(x)=\epsilon n^{-s}(\cos(nx)+\cos((n+1)x).
\end{equation}
Note that $u_0$ is $2\pi$ periodic, smooth, and of mean zero, whence $u_0\in\dot{H}^r(\mathbb{T})$ for any $r\in\mathbb{R}$. A straightforward calculation reveals that 
\[
\|u_0\|_{\dot{H}^r(\mathbb{T})}\sim\epsilon n^{-s+r}\qquad\text{for any $r\in\mathbb{R}$.}
\]
In particular,
\begin{equation}\label{E:u0sT}
\|u_0\|_{\dot{H}^s(\mathbb{T})}\sim \epsilon.
\end{equation}
(But $\|u_0\|_{L^2(\mathbb{T})}$ may be large.)

Recall from the well-posedness theory (see \cite{Kato}, for instance) that a unique solution of \eqref{E:main}-\eqref{def:u0} exists in $C((-t_*,t_*); H^{s_*}(\mathbb{T}))$ for some $t_*>0$, provided that $s_*>3/2$. Let $t_*$ be the maximal time of existence. It follows from the well-posedness theory that 
\[
t_*\gtrsim \|u_0\|_{H^{3/2+0^+}(\mathbb{T})}^{-1}\sim \epsilon^{-1}n^{s-(3/2+0^+)}.
\] 
Since $u_0\in H^\infty(\mathbb{T})$, moreover, it follows from the well-posedness theory that $u\in C((-t_*,t_*); H^\infty(\mathbb{T}))$. Since $u_0$ is of mean zero, it follows from \eqref{E:conservation} that so is $u$ throughout the interval $(-t_*,t_*)$. Therefore, $u\in C((-t_*,t_*); \dot{H}^r(\mathbb{T}))$ for any $r\in\mathbb{R}$. 

Let 
\[
\widehat{S(t)f}(k)=e^{-i|k|^\alpha kt}\widehat{f}(k)\qquad\text{for $k\in\mathbb{Z}$},
\]
and let 
\begin{align}\label{def:u1T}
u_1(x,t)=&S(t)u_0(x) \\
=&\epsilon n^{-s}(\cos(nx-n^{\alpha+1}t)+\cos((n+1)x-(n+1)^{\alpha+1}t)). \notag
\end{align}
Note that $u_1$ solves
\begin{equation}\label{E:u1}
\partial_tu_1+|\partial_x|^\alpha\partial_xu_1=0\quad\text{and}\quad u_1(x,0)=u_0(x).
\end{equation}
In other words, $u_1$ solves the linear part of \eqref{E:main}-\eqref{def:u0}. 
Note that $u_1$ is $2\pi$ periodic, smooth, and of mean zero at any time. A straightforward calculation reveals that
\begin{equation}\label{E:u1rT}
\|u_1(\cdot,t)\|_{\dot{H}^r(\mathbb{T})}\sim\epsilon n^{-s+r}
\qquad\text{for any $t\in\mathbb{R}$}\quad\text{for any $r\in\mathbb{R}$}.
\end{equation} 
In particular, $\|u_1(\cdot,t)\|_{\dot{H}^s(\mathbb{T})}\sim\epsilon$ remains small for any $t\in\mathbb{R}$.

To proceed, let 
\[
u_2(x,t)=-\int^t_0 S(t-\tau)(u_1\partial_xu_1)(x,\tau)~d\tau.
\]
Note that $u_2$ solves
\begin{equation}\label{E:u2}
\partial_tu_2+|\partial_x|^\alpha\partial_xu_2+u_1\partial_xu_1=0\quad\text{and}\quad u_2(x,0)=0.
\end{equation}
As a matter of fact, $u_2$ approximates the solution of \eqref{E:main}-\eqref{def:u0} during some time interval. %We shall show that for $n$ sufficiently large, $\|u_2\|_{\dot{H}^s(\mathbb{T})}$ is large after a short time. 
Note that $u_2$ is $2\pi$ periodic, smooth and of mean zero for any time. 
A straightforward calculation reveals that 
\begin{align*}
u_2(x,t)=
-&\frac12\frac{1}{2^\alpha-1}\epsilon^2n^{-2s-\alpha}
\,\sin((2^\alpha-1)n^{\alpha+1}t)\,\sin(2nx-(2^\alpha+1)n^{\alpha+1}t) \\
-&\frac12\frac{1}{2^\alpha-1}\epsilon^2n^{-2s}(n+1)^{-\alpha} \\
&\times\sin((2^\alpha-1)(n+1)^{\alpha+1}t)\,\sin(2(n+1)x-(2^\alpha+1)(n+1)^{\alpha+1}t) \\
-&\epsilon^2n^{-2s}\frac{1}{(n+1)^{\alpha+1}-n^{\alpha+1}-1} \\
&\times\sin\big(\tfrac12((n+1)^{\alpha+1}-n^{\alpha+1}-1)t\big)
\,\sin\big(x-\tfrac12((n+1)^{\alpha+1}-n^{\alpha+1}+1)t\big)\\
-&\epsilon^2n^{-2s}\frac{2n+1}{(2n+1)^{\alpha+1}-(n+1)^{\alpha+1}-n^{\alpha+1}} \\
&\times\sin\big(\tfrac12((2n+1)^{\alpha}-(n+1)^{\alpha+1}-n^{\alpha+1})t\big)\\
&\times\sin\big((2n+1)x-\tfrac12((2n+1)^{\alpha}+(n+1)^{\alpha+1}+n^{\alpha+1})t\big) \\
\sim-&\frac12\epsilon^2n^{-2s+1}t\,\sin(2nx-(2^{\alpha+1}n^{\alpha+1}t) \\
-&\frac12\epsilon^2n^{-2s}(n+1)t\,\sin(2(n+1)x-(2^{\alpha+1}(n+1)^{\alpha+1}t) \\
-&\frac12\epsilon^2n^{-2s}t\,\sin\big(x-\tfrac12((n+1)^{\alpha+1}-n^{\alpha+1}+1)t\big)\\
-&\frac12\epsilon^2n^{-2s}(2n+1)t\,\sin\big((2n+1)x-\tfrac12((2n+1)^{\alpha}+(n+1)^{\alpha+1}+n^{\alpha+1})t\big) 
\end{align*}
for any $x\in\mathbb{T}$, provided that $0<t\ll n^{-\alpha-1}<1$. Note that the first, the second, and the last terms above have the amplitude of the size $\epsilon^2n^{-2s+1}t$ and the frequency of the size $n$, whereas the third term has the amplitude of the size of $\epsilon^2n^{-2s}t$ and the frequency $1$. In other words, the nonlinear interaction of two adjacent, high frequency modes drives oscillation of a low frequency mode. Consequently, 
\begin{equation}\label{E:u2rT}
\|u_2(\cdot,t)\|_{\dot{H}^r(\mathbb{T})}\sim 
\begin{cases} 
\epsilon^2n^{-2s}t \quad & \text{if $r<-1$}, \\
\epsilon^2n^{-2s+1+r}t & \text{if $r\geq -1$}
\end{cases}
\end{equation}
for any $0<t\ll n^{-\alpha-1}$. We wish to show that for $n$ sufficiently large, $u_2$ becomes large in $\dot{H}^s(\mathbb{T})$ after a short time. As a matter of fact, note that $n^{7s/4-1/2}<n^{-\alpha-1}\ll 1$, by hypothesis, and
\begin{equation}\label{E:u2sT}
\|u_2(\cdot,n^{7s/4-1/2})\|_{\dot{H}^s(\mathbb{T})}\sim\epsilon^2n^{-s/4-1/2}.
\end{equation}
Note that $-s/4-1/2>0$. We may choose $n$ sufficiently large so that $n^{7s/4-1/2}<\epsilon$ and 
\begin{equation}\label{E:n}
\epsilon^2n^{-s/4-1/2}>2\epsilon^{-1}.
\end{equation}

To continue, let 
\[
u=u_1+u_2+w.
\]
Since $u$, and $u_1$, $u_2$ are $2\pi$ periodic, smooth, and of mean zero throughout the interval $(-t_*,t_*)$, so is $w$. Note that $n^{7s/4-1/2}<n^{s-(3/2+0^+)}\lesssim t_*$. We shall show that 
\[
\|w(\cdot,n^{7s/4-1/2})\|_{\dot{H}^0(\mathbb{T})}=\|w(\cdot,n^{7s/4-1/2})\|_{L^2(\mathbb{T})}
\] 
is small. Indeed, $w$ is of mean zero. Consequently,  
\[
\|w(\cdot,n^{7s/4-1/2})\|_{\dot{H}^s(\mathbb{T})}<\|w(\cdot,n^{7s/4-1/2})\|_{\dot{H}^0(\mathbb{T})}
\] 
is small. Note from \eqref{E:main} and \eqref{E:u1}, \eqref{E:u2} that $w$ solves
\[
\partial_tw+|\partial_x|^\alpha\partial_xw+w\partial_xw+\partial_x((u_1+u_2)w)+\partial_x(u_1u_2)+u_2\partial_xu_2=0
\]
and $w(x,0)=0$. Integrating this over $\mathbb{T}$ against $w$, we make an explicit calculation to arrive at
\begin{align*}
\frac12\frac{d}{dt}\|w(\cdot,t)\|_{L^2(\mathbb{T})}^2\leq &
\|\partial_x(u_1+u_2)(\cdot,t)\|_{L^\infty(\mathbb{T})}\|w(\cdot,t)\|_{L^2(\mathbb{T})}^2 \\
&+\|(\partial_x(u_1u_2)+u_2\partial_xu_2)(\cdot,t)\|_{L^2(\mathbb{T})}\|w(\cdot,t)\|_{L^2(\mathbb{T})}
\end{align*}
for any $t\in(-t_*,t_*)$. For $0\leq t\ll n^{s-1}<n^{-\alpha-1}$ so that $n^{-s+1}t\ll1$, note from \eqref{E:u1rT} and \eqref{E:u2rT} that
\begin{align*}
&\|\partial_x(u_1+u_2)(\cdot,t)\|_{L^\infty(\mathbb{T})}\sim 
\epsilon n^{-s+1}+\epsilon^2n^{-2s+2}t\sim \epsilon n^{-s+1}
\intertext{and, similarly,}
&\|(\partial_x(u_1u_2)+u_2\partial_xu_2)(\cdot,t)\|_{L^2(\mathbb{T})}\sim \epsilon^3n^{-3s+2}t.
\end{align*}
It then follows from Gronwall's lemma that
\[
\|w(\cdot,t)\|_{L^2(\mathbb{T})}\lesssim \epsilon^3n^{-3s+2}t^2\exp(\epsilon n^{-s+1}t)
\sim\epsilon^3n^{-3s+2}t^2
\]
for $0\leq t\ll n^{s-1}$, provided that $\epsilon>0$ is sufficiently small. Note that $7s/4-1/2<s-1$ for $s<-2$. Therefore,
\begin{equation}\label{E:w0T}
\|w(\cdot,n^{7s/4-1/2})\|_{L^2(\mathbb{T})}\lesssim \epsilon^3n^{-3s+2}n^{7s/2-1}=\epsilon^3n^{s/2+1}<\epsilon^3.
\end{equation}
 
At last, it follows from the triangle inequality that
\begin{align*}
\|u(\cdot,n^{7s/4-1/2})\|_{\dot{H}^s(\mathbb{T})} 
\geq&\|u_2(\cdot,n^{7s/4-1/2})\|_{\dot{H}^s(\mathbb{T})} \\
&-\|u_1(\cdot,n^{7s/4-1/2})\|_{\dot{H}^s(\mathbb{T})}
-\|w(\cdot,n^{7s/4-1/2})\|_{\dot{H}^s(\mathbb{T})} \\
>&\|u_2(\cdot,n^{7s/4-1/2})\|_{\dot{H}^s(\mathbb{T})} \\
&-\|u_1(\cdot,n^{7s/4-1/2})\|_{\dot{H}^s(\mathbb{T})}
-\|w(\cdot,n^{7s/4-1/2})\|_{L^2(\mathbb{T})} \\
\gtrsim &\epsilon^2n^{-s/4-1/2}-\epsilon-\epsilon^3> \epsilon^{-1},
\end{align*}
provided that $\epsilon>0$ is sufficiently small. This completes the proof. Here, the second inequality uses that $w$ is of mean zero and $s<-2$, the third inequality uses \eqref{E:u2sT}, \eqref{E:u1rT}, and \eqref{E:w0T}, and the last inequality uses \eqref{E:n}.

%%%%%%%%%%%%%%%%%%%%%%%%%%%%%%%%%%%%%%%%%%%%%%%%%%
%%%%%%%%%%%%%%%%%%%%%%%%%%%%%%%%%%%%%%%%%%%%%%%%%%
%%%%%%%%%%%%%%%%%%%%%%%%%%%%%%%%%%%%%%%%%%%%%%%%%%
\subsection*{Acknowledgements}
The author wishes to thank Mimi Dai and Nikos Tzirakis for helpful and stimulating discussions. She is supported by the National Science Foundation under the grant CAREER DMS-1352597, an Alfred P. Sloan Research Fellowship, a Simons Fellowship in Mathematics, and by the University of Illinois at Urbana-Champaign under the Arnold O. Beckman Research Awards RB14100 and RB16227. She is grateful to the Mathematics Department at Brown University for its generous hospitality. 

\bibliographystyle{amsalpha}
\bibliography{fKdV}

\end{document}